\documentclass[leqno]{amsart}
\usepackage{amsmath,amsfonts,amsthm,amssymb,indentfirst,epic,url,graphics,needspace,textcomp}

\newtheorem{theorem}{Theorem}
\newtheorem{lemma}[theorem]{Lemma}
\newtheorem{corollary}[theorem]{Corollary}

\newtheorem{question}[theorem]{Question}

\renewcommand{\r}{\mathrm}
\mathchardef\hy="2D

\begin{document}

\begin{center}
\texttt{Comments, corrections,
and related references welcomed, as always!}\\[.5em]
{\TeX}ed \today
\vspace{2em}
\end{center}

\title{Tensor products of faithful modules}%
\thanks{This preprint is readable online at
% After publication of this note, updates, errata, related references
% etc., if found, will be recorded at
\url{http://math.berkeley.edu/~gbergman/papers/unpub/}
%% and \url{http://arxiv.org/abs/????.?????}\,.
%% The former version is likely to be updated more frequently than
%% the latter.
}

 %%spell.ig
%        ...

\author{George M. Bergman}
\address{Department of Mathematics\\
University of California\\
Berkeley, CA 94720-3840, USA}
\email{gbergman@math.berkeley.edu}

\begin{abstract}
If $k$ is a field, $A$ and $B$ $\!k\!$-algebras,
$M$ a faithful left $\!A\!$-module, and
$N$ a faithful left $\!B\!$-module, we recall the proof that
the left $\!A\otimes_k B\!$-module $M\otimes_k N$ is again faithful.
If $k$ is a general commutative ring, we note some conditions
on $A,$ $B,$ $M$ and $N$ that do, and others that do not, imply the
same conclusion.
Finally, we note a version of the main result that does not
involve any algebra structures on $A$ and $B.$
\end{abstract}

\maketitle
% - - - - - - - - - - - - - - - - - - - - - - - - - - - - - -

I needed Theorem~\ref{T.OX} below, and eventually found a
roundabout proof of it.
Ken Goodearl found a simpler proof, which I simplified further
to the argument given here.
But it seemed implausible that such a result would not be
classical, and I posted a query \cite{query},
to which Benjamin Steinberg responded, noting that
Passman had proved the result in~\cite[Lemma~1.1]{DP}.
His proof is virtually identical to one below.

In the mean time, I had made some observations on
what is true when the base ring $k$ is not a field, and
on the ``irrelevance'' of the algebra structures of $A$ and $B,$
and added these to the write-up.
So while I no longer expect to publish this note,
I will arXiv it, and keep it online as an unpublished note,
to make those observations available.
% But for the moment I'm letting it stay on my ``preprints'' page.

(Ironically, I eventually found an easier way to get
the result for which I had needed Theorem~\ref{T.OX}.)
%kk|so .exrc_tex| " WORK IN PROGRESS ---- p.

\section{The main statement, and a generalization}\label{S.main}

Except where the contrary is stated, we understand algebras
to be associative, but not necessarily unital.

For $k$ a commutative ring, $A$ and $B$ $\!k\!$-algebras,
$M$ a left $\!A\!$-module, and $N$ a left $\!B\!$-module, we
recall the natural structure of left $\!A\otimes_k B\!$-module
on $M\otimes_k N$: An element
\begin{equation}\begin{minipage}[c]{35pc}\label{d.aOXb}
$f\ =\ \sum_{1\leq i\leq n}\ a_i\otimes b_i\ \in\ A\otimes_k B,$
\quad where $a_i\in A,\ b_i\in B,$
\end{minipage}\end{equation}
acts on decomposable elements $u\otimes v$ $(u\in M,\,v\in N)$
of $M\otimes_k N$ by
\begin{equation}\begin{minipage}[c]{35pc}\label{d.aOXb(uOXv)}
$u\otimes v\ \mapsto\ \sum_i\ a_i u\,\otimes\,b_i v.$
\end{minipage}\end{equation}
Since the right-hand side is bilinear in $u$ and $v,$ this map
extends $\!k\!$-linearly to general elements of $M\otimes_k N.$
The resulting action is easily shown to be
compatible with the $\!k\!$-algebra structure of $A\otimes_k B.$

\begin{theorem}\label{T.OX}
Let $k$ be a field, $A$ and $B$ $\!k\!$-algebras,
$M$ a faithful left $\!A\!$-module, and
$N$ a faithful left $\!B\!$-module.
Then the left $\!A\otimes_k B\!$-module $M\otimes_k N$
is also faithful.
\end{theorem}

\begin{proof}[Proof \textup{(after K.\,Goodearl, D.\,Passman).}]
Given nonzero $f\in A\otimes_k B,$
we wish to show that it has nonzero action on $M\otimes_k N.$
Clearly, we can choose an expression~\eqref{d.aOXb} for $f$
such that the $b_i$ are $\!k\!$-linearly independent.
(We could simultaneously make the $a_i$ $\!k\!$-linearly
independent, but will not need to.)
Since we have assumed~\eqref{d.aOXb} nonzero, not all of
the $a_i$ are zero; so as $M$ is a
faithful $\!A\!$-module, we can find $u\in M$ such
that not all the $a_i u\in M$ are zero.
Hence there exists a $\!k\!$-linear
functional $\varphi: M\to k$ such that not all of the
$\varphi(a_i u)$ are zero.
Since the $b_i$ are $\!k\!$-linearly independent, the element
$\sum_i \varphi(a_i u)\,b_i\in B$ will thus be nonzero.
So as $N$ is a faithful $\!B\!$-module, we can choose $v\in N$ such that
\begin{equation}\begin{minipage}[c]{35pc}\label{d.neq0}
$(\,\sum_i\,\varphi(a_i u)\,b_i)\,v\ \neq\ 0$\quad in $N.$
\end{minipage}\end{equation}

We claim that for the above choices of $u$ and $v,$ if we apply $f$ to
$u\otimes v\in M\otimes_k N,$ the result, i.e.,
the right-hand side of~\eqref{d.aOXb(uOXv)}, is nonzero.
For if we apply to that element the map $\varphi\otimes\r{id}_N:
M\otimes_k N\to k\otimes_k N\cong N,$ we get the
nonzero element~\eqref{d.neq0}.
Thus, as required, $f$ has nonzero action on $M\otimes_k N.$
\end{proof}

The above result assumes $k$ a field.
Succumbing to the temptation to examine what the method
of proof can be made to give in the absence of that
hypothesis, we record

\begin{corollary}[to proof of Theorem~\ref{T.OX}]\label{C.OX}
Let $k$ be a commutative ring, $A$ and $B$ $\!k\!$-algebras,
$M$ a faithful left $\!A\!$-module, and
$N$ a faithful left $\!B\!$-module.

Suppose, moreover, that elements of $M$
can be separated by $\!k\!$-module homomorphisms $M\to k,$ and that
every finite subset of $B$ belongs to a free $\!k\!$-submodule of~$B.$

Then the left $\!A\otimes_k B\!$-module $M\otimes_k N$
is again faithful.
\end{corollary}

\begin{proof}
Exactly like the proof of Theorem~\ref{T.OX}.
The added hypothesis on $B$ is what we need to conclude that any
element of $A\otimes_k B$ can be written in the form~\eqref{d.aOXb}
with $\!k\!$-linearly independent $b_i;$ the added hypothesis
on $M$ is what we need to construct $\varphi.$
(Of course, the parenthetical comment in the proof of Theorem~\ref{T.OX}
about also making the
$a_i$ $\!k\!$-linearly independent does not go over.)
\end{proof}

Warren Dicks (personal communication) pointed
out early on another proof of Theorem~\ref{T.OX}:
The actions of $A$ and $B$ on $M$ and $N$ yield embeddings
$A\to\r{End}_k(M)$ and $B\to\r{End}_k(N).$
Taking $\!k\!$-bases $X$ and $Y$ of $M$ and $N,$ one
can regard the underlying vector spaces of
$\r{End}_k(M)$ and $\r{End}_k(N)$ as $M^X$ and $N^Y.$
By two applications of \cite[II, \S3, \textnumero 7,
Cor.~3 to Prop.~7, p.AII.63]{Bourbaki},
or one of \cite[Theorem~2]{KG},
one concludes that the natural map
$M^X\otimes_k N^Y\to(M\otimes_k N)^{X\times Y}$
is an embedding, and deduces that the map
$A\otimes_k B\to\r{End}(M\otimes_k N)$ is an embedding,
the desired conclusion.

\section{Counterexamples to variant statements}\label{S.cegs}

The hypotheses of the above corollary are strikingly asymmetric in the
pairs $(A,M)$ and $(B,N).$
We can, of course, get the same conclusion if we interchange the
assumptions on these pairs.
But what if we try to use one or the other hypothesis
on both pairs; or concentrate both hypotheses on one of them?

It turns out that none of these modified hypotheses
guarantees the stated conclusion.
Here are three closely related constructions that give
the relevant counterexamples.
In all three examples, the $\!k\!$-algebras $A$ and $B$ are
in fact commutative and unital.

\begin{lemma}\label{L.OX.cegs}
Let $C$ be a commutative principal ideal domain with infinitely
many primes ideals, and let $P_0,$ $P_1$ be two disjoint infinite
sets of nonzero prime ideals of $C.$
Then for $k,$ $A,$ $M,$ $B,$ $N$ specified in each of the following
three ways, the $\!A\!$-module $M$ and the $\!B\!$-module $N$
are faithful, and $A\otimes_k B$ is nonzero, but
$M\otimes_k N$ is zero, hence not faithful.
In each example, the variant of the hypotheses of Corollary~\ref{C.OX}
satisfied by that example is noted at the end of the description.
\vspace{.2em}

\textup{(i)}  Let $A=B=k=C,$ and let
$M=\bigoplus_{p\in P_0} k/p$ and $N=\bigoplus_{p\in P_1} k/p.$
In this case, every finite subset of $A$ or of $B$
\textup{(}trivially\textup{)}
lies in a free $\!k\!$-submodule of that algebra.
\vspace{.2em}

In the remaining two examples,
let $C^+$ be the commutative $\!C\!$-algebra obtained by adjoining
to $C$ an indeterminate $y_p$ for each $p\in P_0\cup P_1,$
and imposing the relations $p\,y_p=\{0\}$ for each such $p.$
\vspace{.2em}

\textup{(ii)}  Let $k=C^+,$
let $M$ be the ideal of $k$ generated by the $y_p$ for $p\in P_0,$
let $N$ be the ideal of $k$ generated by the $y_p$ for $p\in P_1,$
let $A=k/N,$ and let $B=k/M.$
Note that since $MN=\{0\},$
we may regard $M$ as an $\!A\!$-module and $N$ as a $\!B\!$-module.
In this case, $M$ and $N$ embed in $k,$ hence
$\!k\!$-module homomorphisms from each of those
modules into $k$ separate elements.
\vspace{.2em}

\textup{(iii)}  Let $k=A=C^+,$ and let $M$ be the ideal
of $A$ generated by all the $y_p.$
\textup{(}The partition of our infinite family of primes
into $P_0$ and $P_1$ will not be used here.\textup{)}
On the other hand, let $B$ be the
field of fractions of $C,$ made a $\!k\!$-algebra by
first mapping $k$ to $C$ by sending the $y_p$ to $0,$
then mapping $C$ to its field of fractions;
and let $N$ be any nonzero $\!B\!$-vector-space.
In this case, every finite subset of $A$
\textup{(}trivially\textup{)} lies in a free $\!k\!$-submodule, and
$\!k\!$-module homomorphisms into $k$ clearly separate elements of $M.$
\end{lemma}

\begin{proof}
The fact that $P_0$ and $P_1$
are infinite sets of primes in the commutative principal
ideal domain $C$ implies that each of those sets has zero intersection,
hence that the $M$ and $N$ of~(i) are faithful $\!C\!$-modules,
equivalently, are a faithful $\!A\!$-module and a
faithful $\!B\!$-module.
That their tensor product is zero is clear.

Similar considerations show that in~(ii), any element
of $A$ or $B$ with nonzero constant term in $C$ acts
nontrivially on $M,$ respectively, $N.$
On the other hand, a nonzero element of $A$ or $B$ with
zero constant term, i.e., a nonzero element $u$ of the
ideal $M$ or $N,$ will act nontrivially on the
module $M,$ respectively, $N,$ because for some $p$ in $P_0,$
respectively $P_1,$ the element $u$ must involve a nonzero
polynomial in $y_p,$ which will have nonzero action on $y_p\in M$
or $N$ as the case may be.
(Since $y_p\,y_{p'}=\{0\}$ for $p\neq p',$ every element of $M$ or $N$
is a sum of one-variable polynomials in the various $y_p$
with zero constant term.)
Again, it is clear that $M\otimes_k N=\{0\}.$
On the other hand,
$A\otimes_k B = (k/M)\otimes_k (k/N)\cong k/(M+N)\cong C\neq\{0\}.$

In case~(iii), $M$ is faithful over $A$ for the same
reason as in~(ii), while faithfulness of $N$ over $B$ is clear,
as is the condition $M\otimes_k N=\{0\}.$
On the other hand, since $A$ admits a homomorphism into $C,$
we have $A\otimes_k B\neq\{0\}.$
\end{proof}

We remark that in the above
constructions, the condition that the principal ideal domain $k$
have infinitely many primes can be weakened to say that it
has at least two primes, $p_0$ and $p_1,$ by replacing
the modules $\bigoplus_{p\in P_0} k/p$
and $\bigoplus_{p\in P_1} k/p$ in~(i) with
$\bigoplus_{n>0} k/p_0^n$ and $\bigoplus_{n>0} k/p_1^n,$
and making similar adjustments in~(ii) and~(iii).
(In the last, only one prime is needed.)
One just has to be a little more careful in the proofs.

In another direction, one may ask:

\begin{question}[suggested by A.\,Ogus, personal communication]\label{Q.OX}
If we add to the hypotheses of Corollary~\ref{C.OX}
the condition that $M$ be finitely generated
as an $\!A\!$-module, and/or that $N$ be finitely generated
as a $\!B\!$-module, can some of the other hypotheses of that
corollary be weakened, dropped, or modified
\textup{(}perhaps in some of the ways that Lemma~\ref{L.OX.cegs}
shows is {\em not} possible without such finite
generation conditions\textup{)}?
\end{question}

\section{$A$ and $B$ don't have to be algebras}\label{S.nonalg}

The sharp-eyed reader may have noticed that the proof of
Theorem~\ref{T.OX} makes no use of the algebra structures
of $A$ and $B.$
This led me to wonder whether the result was actually a special case
of a statement that involved no such structure.
As Theorem~\ref{T.nonalg} below shows, the answer is, in a way, yes.
But as the second proof of that theorem shows, one can equally regard
Theorem~\ref{T.nonalg} as a special case of Theorem~\ref{T.OX}.

Given a commutative ring $k,$ and $\!k\!$-modules
$A,$ $M_0$ and $M_1,$ let us define an {\em action}
of $A$ on $(M_0,M_1)$ to mean
a $\!k\!$-linear map $A\to\r{Hom}_k(M_0,M_1),$ which will be written
$(a,u)\mapsto au$ $(a\in A,\ u\in M_0,\ au\in M_1);$
and let us call such an action {\em faithful} if it is one-to-one
as a map $A\to\r{Hom}(M_0,M_1).$

Given two actions, one of a $\!k\!$-module
$A$ on a pair $(M_0,M_1)$ and the other
of a $\!k\!$-module $B$ on a pair $(N_0,N_1),$
we see that an action of $A\otimes_k B$ on
$(M_0\otimes_k N_0,\,M_1\otimes_k N_1)$
can be defined just as for algebras, with
each element~\eqref{d.aOXb} acting by~\eqref{d.aOXb(uOXv)}.

\begin{theorem}\label{T.nonalg}
Let $k$ be a field, and suppose we are given
an action of a $\!k\!$-vector-space $A$
on a pair $(M_0,M_1)$ and an action
of a $\!k\!$-vector-space $B$ on a pair $(N_0,N_1).$

Then if each of these actions is faithful, so is the
induced action of the $\!k\!$-vector-space $A\otimes_k B$ on the pair
$(M_0\otimes_k N_0,\,M_1\otimes_k N_1).$
\end{theorem}

\begin{proof}[First proof.]
Exactly like proof of Theorem~\ref{T.OX}.
(Note that $u$ will be chosen from $M_0,$ while
$\varphi$ will be a $\!k\!$-linear functional on $M_1;$ and that
$v$ will be chosen from $N_0$ to make~\eqref{d.neq0}
hold in $N_1.)$\\[.2em]
{\em Second proof.}
Let us make $A$ and $B$ into $\!k\!$-algebras by giving them
zero multiplication operations.
Then we can make the vector space $M=M_0\oplus M_1$ a left
$\!A\!$-module using the action $a(u_0,u_1)=(0,a\,u_0),$
and similarly make $N=N_0\oplus N_1$ a $\!B\!$-module.
The faithfulness hypotheses on the given actions clearly
make these modules faithful.

Hence by Theorem~\ref{T.OX}, $M\otimes_k N$ is
a faithful $\!A\otimes_k B\!$-module.
Now $M\otimes_k N$ is a fourfold direct sum
$(M_0\otimes_k N_0)\oplus (M_0\otimes_k N_1)\oplus
(M_1\otimes_k N_0)\oplus (M_1\otimes_k N_1);$
but the action of $A\otimes_k B$ annihilates all
summands but the first, and has image in the last, so
its faithfulness means that every nonzero element
of $A\otimes_k B$ induces a nonzero map from
$M_0\otimes_k N_0$ to $M_1\otimes_k N_1,$
which is the desired conclusion.
\end{proof}

Of course, the analog of Corollary~\ref{C.OX}
holds for actions as in Theorem~\ref{T.nonalg}.

\section{Remarks}\label{S.remarks}
Composing the proof of Theorem~\ref{T.nonalg}
from Theorem~\ref{T.OX}, and the proof of Theorem~\ref{T.OX}
from Theorem~\ref{T.nonalg},
we see that the general case of Theorem~\ref{T.OX}
follows from the zero-multiplication case of the same theorem.
This is striking, since the main interest of the result is for
algebras with nonzero multiplication.

One may ask why I made the convention that
algebras are associative, if the algebra
operations were not used in the theorem.
The answer is that there is no natural definition of
a module over a not-necessarily-associative algebra.
(There is a definition of a module over a Lie algebra,
based on the motivating relation between Lie algebras
and associative algebras.
But there is no natural definition
of a Lie or associative structure on a tensor product of Lie
algebras, so the result can't be used in that case.)

Which of Theorems~\ref{T.OX} and~\ref{T.nonalg} is the ``nicer''
result?
I would say that Theorem~\ref{T.nonalg} shows with less
distraction what is going on, while Theorem~\ref{T.OX}
is likely to be more convenient for applications.

\end{document}